\documentclass[11pt]{amsart} 
\usepackage{amssymb,amsfonts,graphicx,amsmath,amsthm,amstext,latexsym,amsfonts}
\usepackage{graphicx,epsfig,color}
\usepackage{bbm}
\usepackage{t1enc}
\usepackage{mathrsfs}
\usepackage{subfig}
\usepackage{hyperref}
\usepackage{a4wide}
\usepackage[T1]{fontenc}

\usepackage{graphicx}        
\usepackage{tikz}
 \usetikzlibrary{arrows}
 \usetikzlibrary{calc}
\usepackage{pgfplots}
\usetikzlibrary{intersections}
\theoremstyle{plain}

\numberwithin{equation}{section}
\newtheorem{theorem}{Theorem}[section]
\newtheorem{lemma}[theorem]{Lemma}

\usepackage{color}
\definecolor{darkred}{rgb}{0.8,0,0}
\definecolor{darkblue}{rgb}{0,0,0.7}
\definecolor{darkgreen}{rgb}{0,0.4,0}

\newcommand{\KKK}{\color{black}}

\newcommand{\R}{{\mathbb R}}

\newcommand{\un}{{\rm 1\kern -2.5pt l}}

\newcommand{\xxi}{{\mbox{\boldmath$\xi$}}}

\newcommand{\pphi}{{\mbox{\boldmath$\varphi$}}}
\newcommand{\ppsi}{{\mbox{\boldmath$\psi$}}}
\def\u{\mathbf{u}}
\def\uu{\mathbf{uu}}

\def\yy{\mathbf{y}}

\def\xxi{\boldsymbol{\xi}}
\def\uu{\mathbf{u}}

\def\R{{\mathbb R}}

\def\R{{\mathbb R}}

\def\spt{\mathop{{\rm spt}}\nolimits}

\def\dv{\mathop{{\rm div}}\nolimits}

\def\y{\mathbf y}
\def\u{\mathbf{u}}
\def\v{\mathbf{v}}

\def\v{{\bf v}}

\def\x{{\bf x}}

\def\wconv{\rightharpoonup}

\renewcommand{\epsilon}{\varepsilon}
\newcommand{\beeq}{\begin{equation}}
\newcommand{\eneq}{\end{equation}}
\newcommand{\bear}{\begin{array}}
\newcommand{\enar}{\end{array}}
\newcommand{\bema}{\begin{displaymath}}
\newcommand{\enma}{\end{displaymath}}

\newcommand{\beea}{\begin{eqnarray}}
\newcommand{\enea}{\end{eqnarray}}

\newcommand{\om}{\Omega}

\newcommand{\lab}[1]{ \label{#1} }

\def\wconv{\rightharpoonup}

\title[]{Newton's second law  as  limit of 
variational problems}
   \author[]{Edoardo Mainini and Danilo Percivale}
 \address{ Edoardo Mainini and Danilo Percivale} 
  \address{ Dipartimento di   Ingegneria Meccanica, Energetica, Gestionale e dei Trasporti (DIME), \vspace{-0.3cm}} 
 \address{Universit\`{a} degli Studi di Genova,
 Via all'Opera Pia, 15 - 16145 Genova Italy.\bigskip}

  \email{edoardo.mainini@unige.it;  percivale@dime.unige.it}

\date{}  
\begin{document}
 \maketitle
\begin{abstract} \vspace{-0.5cm}
We show that the solution of Cauchy problem for the classical ode $m \mathbf y''=\mathbf f$ can be obtained as limit of minimizers of exponentially weighted convex variational integrals. This complements the known results about weighted inertia-energy  approach to Lagrangian mechanics and hyperbolic equations.
\end{abstract}
\begin{center}
\end{center}
\vskip0.5cm
\section{Introduction and statement of the result}

\let\thefootnote\relax\footnotetext{\noindent\!\!\!\!\!\!\!\!\! 
 {\it 2020 AMS Classification Numbers}: 49J45, 70H30.	\\
 {\it Key words and phrases}: Calculus of Variations, Newton's second law,
 Weighted Variational Integrals.}
  
  

Let $\mathbf f\in L^\infty(\R^+;\R^N)$,  $\u_0\in\R^N$, $\v_0\in\R^N$,  $m>0$. Let us consider the Cauchy problem
\beeq\lab{newton0}\left\{\begin{array}{ll}\displaystyle & m{\y''}=\mathbf f,\qquad t>0\\
&\\
& \y(0)=\u_0,\qquad { \y'}(0)=\v_0
\end{array}\right.
\eneq
governing the motion of a material point of mass $m$ subject to the force field $\mathbf f$.
Our goal is to show that the solution to \eqref{newton0} is the limit as $h\to+\infty$ of the minimizers of the following functionals defined on trajectories $\mathbf y:\R^+\to\R^N$
\[\displaystyle \frac m{2h^2}\int_0^{+\infty} |\y''(t)|^2 e^{-ht}\,dt- \int_0^{+\infty} \mathbf f_h(t) \cdot \y(t) e^{-ht}\,dt, \qquad h\in\mathbb N,\]
%
subject to the same initial conditions,
as soon as $(\mathbf f_h)_{h\in\mathbb N}\subset L^\infty(\mathbb R^+;\R^N)$ is a sequence such that $\mathbf f_h\wconv \mathbf f$ in $w^*-L^\infty(\R^+;\R^N)$ as $h\to+\infty$. 
More precisely, letting 
\begin{equation*}
\mathcal A:=\left\{\v\in W^{2,1}_{loc}(\R^+;\R^N): \int_0^{+\infty}|\v''(t)|^2 e^{-t}\,dt < +\infty\right\}
\end{equation*} 
and  $\mathbf f_h\in L^\infty(\mathbb R^+;\R^N)$ for every $h\in\mathbb N$, we may define the rescaled energy functional (see also Lemma \ref{scaling} below)
 \begin{equation*}
\mathcal J_h(\u):=\left\{\begin{array}{ll}&\displaystyle\frac{m}{2}\int_0^{+\infty}|\u''(t)|^2\,e^{-t}\,dt-
h^{-2}\int_0^{+\infty}\mathbf f_h(h^{-1}t)\cdot\u(t)\,e^{-t}\,dt\quad \text{if}\ \uu \in\mathcal A\\
&\\
& +\infty\quad \text{otherwise in}\ W^{2,1}_{loc}(\R^+;\R^N),
\end{array}\right.
\end{equation*}
and we will prove the following result.
\begin{theorem}\lab{main} For every $h\in\mathbb N$, there exists a unique solution $\overline\u_h$ to the problem $$min\{\mathcal J_h(\u): \u\in \mathcal A,\ \u(0)=\u_0,\ \u' (0)=h^{-1}\v_0\}.$$
Moreover, if $\mathbf f_h\wconv\mathbf f$ in $w^*-L^\infty(\R^+;\R^N)$ as $h\to+\infty$,  by setting  $\overline\yy_h(t):=\overline\u_h(ht)$ we have $\overline\yy_h\wconv \overline\yy$ in $w^*-W^{2,\infty}((0,T); \R^N)$ for every $T > 0$,
where  $\overline\yy$ is the unique solution on $\R^+$ of problem \eqref{newton0}. 
\end{theorem}


A variational approach based on the minimization of \textit{weighted inertia-energy} (WIE) functionals can be used for approximating large classes of initial value problems of the second order.
An example is the nonhomogeneous wave equation $$w_{tt}=\Delta w +g \qquad\mbox{in }  \R^+\times\R^N.$$ Indeed, it has been shown in \cite{TT1} that given $g\in L^2_{loc}((0,+\infty);L^2(\mathbb R^N))$, $\alpha\in H^1(\R^N)$, $\beta\in H^1(\R^N)$, there exists a sequence $(g_h)_{h\in\mathbb N}$ converging  to $g$ in $L^2((0,T);L^2(\mathbb R^N))$ for every $T>0$ such that the following properties hold. First, the WIE functional 
\[
\int_0^{+\infty}\int_{\R^N}e^{-t}\left\{\frac12|u_{tt}(t,x)|^2+\frac12h^{-2}|\nabla u(t,x)|^2-h^{-2}g_h(h^{-1}t,x)u(t,x)\right\}\,dt\,dx
\]
has, for every $h\in\mathbb N$, a unique minimizer $u_h$
in the class of functions $u\in L^{1}_{loc}(\R^+\times\R^N)$ such that 
\[ \left\{\begin{array}{ll} &\nabla u\in L^{1}_{loc}(\R^+\times\R^N),\qquad u''\in L^{1}_{loc}(\R^+\times\R^N),\\
&\\
&\displaystyle \int_0^{+\infty}\!\!\int_{\R^N}e^{-t}\left\{|u_{tt}|^2+|\nabla u|^2\right\}\,dt\,dx < +\infty,\\
&\\
& u(0,x)=\alpha (x),\qquad u_{t}(0,x)=h^{-1}\beta (x).\\
\end{array}\right.
\]
Second,
 by setting $ w_h(t,x):=u_h(ht,x)$,
the sequence $(w_h)_{h\in\mathbb N}$ converges
 weakly in $H^1((0,T)\times \R^N)$ for every $T>0$ to a function $w$ which solves in the sense of distributions in $\R^+\times \R^N$ the initial value problem
 \beeq\label{nonhomogeneous} \left\{\begin{array}{ll} &w_{tt}= \Delta w+g\\
 &\\
 & w(0,x)=\alpha (x),\qquad w_{t}(0,x)=\beta (x).
 \end{array}\right.
 \eneq
A similar result holds true for other classes of hyperbolic equations as shown in \cite{ST2, TT2, LS2}. In particular it applies to the nonlinear wave equation $w_{tt}=\Delta w-\tfrac p2|w|^{p-2}w$, $p\ge 2$, as  conjectured by De Giorgi \cite{DG} and first proven in \cite{ST}, see also \cite{S}.
Let us mention that (the scalar version of) Theorem \ref{main} is not a direct consequence of the above result from \cite{TT1}, since one should apply the latter to constant-in-space forcing terms $g$ and initial data $\alpha,\beta$, and since the approximating sequence $(g_h)_{h\in\mathbb N}$ in \cite{TT1} is not arbitrary but  obtained by means of a specific construction, not allowing for instance for the choice $g_h\equiv g$ for every $h$.

Concerning the WIE approach for odes, let us mention its application in \cite{LS} for providing  
a variational approach to Lagrangian mechanics, by considering an equation of the form  \beeq\label{potentialenergy} m \mathbf y''+\nabla U(\mathbf y)=0,\qquad t>0\eneq for given potential energy $U\in C^1(\mathbb R^N)$, bounded from below, and $m>0$. The main theorem of \cite{LS} proves indeed that solutions to the initial value problem for \eqref{potentialenergy} can be approximated by rescaled minimizers, subject to the same initial conditions, of the functionals
\[
\mathcal G_h(\v)=\int_0^{+\infty}e^{-t}\left\{\frac m2|\mathbf v''(t)|^2+h^{-2} U(\mathbf v(t))\right\}\,dt,\qquad h\in\mathbb N.
\]
{It is worth noticing that, also in this case, Theorem \ref{main} is not a  consequence of the result from \cite{LS} since 
the latter requires that the force field is conservative  and independent of $t$.}

We have already observed that in the scalar case problem \eqref{newton0} is a particular case of problem \eqref{nonhomogeneous}, obtained by taking constant initial data and letting the forcing term depend only on time. Let us also mention 
another interpretation of \eqref{newton0}
from a continuum mechanics point of view. Indeed,  Newton's second law \eqref{newton0} governs the motion of the center of mass of a body occupying a reference configuration $\Omega\subset\R^N$. More in detail, let 
$\rho$ be the mass density of the body and let  
$\u(t,\x)$ be the position of the material point $\x$ at time  $t$. If $\mathbb T$ is the Cauchy stress  tensor and  $\mathbf b$ is the body force field acting on $\om$, then the equation of motion,  see for instance \cite{G},
 takes the form
\beeq\lab{motion}\rho \u_{tt}= \dv \mathbb T+ \mathbf b\qquad\mbox{ in $\R^+\times\Omega.$}\eneq
Therefore,   by integrating in $\om$ both sides of \eqref{motion}, we formally get
\[\displaystyle \frac{d^2}{dt^2}\left (\int_{\om}\rho \u\, d\x\right )= \int_{\om} \dv \mathbb T\, d\x+ \int_{\om}\mathbf b\,d\x=
\int_{\partial\om} \mathbb T\cdot{\mathbf n}\, d\mathcal H^{N-1}+ \int_{\om}\mathbf b\,d\x=: \mathbf f_\om,\qquad t>0,\]
that is, 
\[ m_\om\,  \y''= \mathbf f_\om,\qquad t>0,\]
where $ \mathbf f_\om=\mathbf f_\om(t)$ is the total force acting on the body, accounting for surface and body forces, 
$  m_\om=\int_{\om}\rho(\x)\,d\x$
is the mass of the body and
$$\y(t)=m_\om^{-1}\int_{\om}\rho(\x)\, \u(t,\x)\,d\x $$
is the position at time t of the center of mass of the body during the motion. Therefore Newton's second law \eqref{newton0} can be viewed as the average in space 
of the equation of motion \eqref{motion}. 
In this perspective Theorem \ref{main}  can be seen as a result about the  equation of motion in $\R^N$ in the above average sense.

Let us finally stress that the methods that are described in this paper, here  only devoted to the elementary problem \eqref{newton0}, can be extended to nonlinear problems like $\mathbf y''=\nabla_{\mathbf y} G(t,\mathbf y)$ under suitable assumptions on $G$, but also to hyperbolic problems such as \eqref{nonhomogeneous} allowing to get further results on these topics. In this perspective, we will develop our analysis in a forthcoming paper.

\KKK

 \section{Existence of minimizers}
 In this section we provide some preliminary results that we are going to use for proving  Theorem \ref{main}.
First of all, it is worth noticing that if $\u\in \mathcal A$ then $\u\in W^{2,2}((0,T);\R^N)$ for every $T > 0$ hence both $\u(0)$ and $\u'(0)$ are well defined.  Moreover, if $\u\in \mathcal A$, by Cauchy-Schwarz inequality
\[
\left|\int_0^{+\infty}\mathbf f(h^{-1}t)\cdot\u(t)e^{-t}\,dt\right|\le\|\mathbf f\|_{\infty}\left(\int_0^{+\infty}|\mathbf u(t)|^2e^{-t}\,dt\right)^{1/2}
\] 
and the integral in the left hand side is finite (see Lemma \ref{newlemma} below),
so that $\mathcal J_h(\u)$ is well-defined and finite.  \KKK
In fact, we have the following estimates

\begin{lemma} \label{newlemma}
 Let $\u\in\mathcal A$. Then $e^{-t/2}\u\in L^2((0,+\infty);\R^N)$, $e^{-t/2}\u'\in L^2((0,+\infty);\R^N)$ and
\begin{equation}\label{thefirst}
\int_0^{+\infty}|\u'(t)|^2e^{-t}\,dt\le 2|\u'(0)|^2+4\int_0^{+\infty}|\u''(t)|^2e^{-t}\,dt,
\end{equation}
\begin{equation}\label{thesecond}
\int_0^{+\infty}|\u(t)|^2e^{-t}\,dt\le 2|\u(0)|^2+8|\u'(0)|^2+16\int_0^{+\infty}|\u''(t)|^2e^{-t}\,dt.
\end{equation}
\end{lemma}
\begin{proof}
 We have $\u\in AC([0,T];\mathbb R^N)$ and $\u'\in AC([0,T];\R^N)$ for every $T>0$.  Therefore $\frac {d}{dt}|\u(t)|^2=2\u(t)\cdot\u'(t)$ and $\frac {d}{dt}|\u'(t)|^2=2\u'(t)\cdot\u''(t)$ for a.e. $t>0$. Moreover, given $T>0$ we  integrate by parts and obtain
\[\begin{aligned}
\int_0^{T}|\u'(t)|^2e^{-t}\,dt&=\left[-e^{-t}|\u'(t)|^2\right]_0^T+2\int_0^Te^{-t/2}\u'(t)\cdot\u''(t)e^{-t/2}\,dt\\&\le |\u'(0)|^2+\frac12\int_0^T|\u'(t)|^2e^{-t}\,dt+2\int_0^T|\u''(t)|^2e^{-t}\,dt,
\end{aligned}\]
where we have used Young inequality.
By letting $T\to+\infty$ we get \eqref{thefirst}. The same computation entails
\[
\int_0^{T}|\u(t)|^2e^{-t}\,dt\le |\u(0)|^2+\frac12\int_0^T|\u(t)|^2e^{-t}\,dt+2\int_0^T|\u'(t)|^2e^{-t}\,dt.
\]
By letting $T\to+\infty$ and by taking advantage of \eqref{thefirst} we obtain \eqref{thesecond}. 
\end{proof}

\KKK

The next lemma proves the first statement of Theorem \ref{main}.
\begin{lemma}
For every $h\in\mathbb N$ there exists a unique solution to the problem
\begin{equation}\label{pb1}
\min \{\mathcal J_h(\u): \u\in\mathcal A,\, \u(0)=\u_0, \u'(0)=h^{-1}\v_0\}
\end{equation}
\end{lemma}
\begin{proof}
We first observe  that $\mathcal J_h$ is strictly convex and that the minimization set is convex. Therefore if a minimizer exists it is necessarily unique, so we are left to prove existence.
If $\u\in\mathcal A$ is such that $\u(0)=\u_0,\ \u' (0)=h^{-1}\v_0$, Lemma \ref{newlemma} entails
\beeq\lab{est1}
\int_0^{+\infty}|\u(t)|^2e^{-t}\,dt\le 2|\u_0|^2+8h^{-2}|\v_0|^2+ 16\int_0^{+\infty}|\u''(t)|^2e^{-t}\,dt
\eneq and
\beeq\lab{est1bis}
\int_0^{+\infty}|\u'(t)|^2e^{-t}\,dt\le 2h^{-2}|\v_0|^2+ 4\int_0^{+\infty}|\u''(t)|^2e^{-t}\,dt.
\eneq
Let $(\u_k)_{k\in\mathbb N}$ be a minimizing sequence for problem \eqref{pb1}. Since $\u_0+h^{-1}t\v_0$ is admissible for problem \eqref{pb1},  we have for any large enough $k$
\[
\mathcal J_h(\u_k)\le \mathcal J_h(\u_0+h^{-1}t\v_0)+1,\]
whence by \eqref{est1}, by Young and Cauchy-Schwarz inequalities, and by denoting with $C$ various constants only depending on $\|\mathbf f_h\|_{\infty}, h, \u_0, \v_0$, $m$, we get
\beeq\lab{est3}\begin{aligned}
&\displaystyle\int_0^{+\infty}|\u_k''(t)|^2e^{-t}\,dt\le \frac{2}{m}h^{-2}\int_0^{+\infty}\mathbf f_h(h^{-1}t)\cdot\u_k(t)e^{-t}\,dt\\&\qquad\qquad\qquad\qquad\qquad-
\frac2 m h^{-2}\int_0^{+\infty}\mathbf f_h(h^{-1}t)\cdot (\u_0+h^{-1}t\v_0)e^{-t}\,dt+ \frac2m\\
&\quad\displaystyle\le \frac{2}{m}\|\mathbf f_h\|_{\infty}h^{-2}\int_0^{+\infty}|\u_k(t)|e^{-t}\,dt+C\le \frac{2}{m}\|\mathbf f_h\|_{\infty}h^{-2}\left (\int_0^{+\infty}|\u_k(t)|^2 e^{-t}\,dt\right )^{\frac{1}{2}}+C\\
&\quad\displaystyle \le \frac{1}{32}\int_0^{+\infty}|\u_k(t)|^2 e^{-t}\,dt + \frac{32}{m^2}h^{-4}\|\mathbf f_h\|^2_{\infty}+C
\le \frac{1}{2}\int_0^{+\infty}|\u_k''(t)|^2e^{-t}\,dt + C.
\end{aligned}
\eneq
By taking into account of \eqref{est1}, \eqref{est1bis}, \eqref{est3} we get that the sequence $(e^{-\frac{t}{2}}\u_k)_{k\in\mathbb N}$ is equibounded in 
$W^{2,2}(\R^+;\R^N)$, so there exists $\v\in W^{2,2}(\R^+;\R^N)$ such that up to extracting a subsequence there holds $e^{-\frac{t}{2}}\u_k\wconv\v$ in $W^{2,2}(\R^+;\R^N)$, hence $\u_k\wconv \u:=e^{\frac{t}{2}}\v$
in $W^{2,2}((0,T);\R^N)$ for every $T > 0$ and  $\u(0)=\u_0$, $\u'(0)=h^{-1}\v_0$. \KKK Therefore we have for every $T > 0$
\[
\displaystyle\liminf_{k\to +\infty}\int_0^{+\infty}|\u_k''(t)|^2e^{-t}\,dt \ge \liminf_{k\to +\infty}\int_0^{T}|\u_k''(t)|^2e^{-t}\,dt\ge \int_0^{T}|\u''(t)|^2e^{-t}\,dt,
\]
hence
\[
\displaystyle\int_0^{+\infty}|\u''(t)|^2e^{-t}\,dt= \sup_{T> 0}\int_0^{T}|\u''(t)|^2e^{-t}\,dt\le \liminf_{k\to +\infty}\int_0^{+\infty}|\u_k''(t)|^2e^{-t}\,dt,
\]
so eventually we find $\u\in \mathcal A$,
and since
\[\begin{aligned}
\displaystyle\lim_{k\to +\infty}\int_0^{+\infty}h^{-2}\mathbf f_h(h^{-1}t)\cdot \u_k \,e^{-t}\,dt&=
 \int_0^{+\infty}h^{-2}\mathbf f_h(h^{-1}t)\cdot \v \,e^{-t/2}\,dt\\&=\int_0^{+\infty}h^{-2}\mathbf f_h(h^{-1}t)\cdot \u\, e^{-t}\,dt
\end{aligned}
\]
 we get
\[
\liminf_{k\to +\infty}\mathcal J_h(\u_k)\ge \mathcal J_h(\u).
\]
We conclude that $\u$ is solution to \eqref{pb1}.
\end{proof}
\begin{lemma}\lab{scaling} Let $h\in\mathbb N$. If $\overline\u_h$ is the unique solution to \eqref{pb1}, then $\overline\y_h(t):=\overline\u_h(ht)$
is the unique minimizer of
\[
\mathcal F_h(\y):=\left\{\begin{array}{ll}&\displaystyle\frac{m}{2h^2}\int_0^{+\infty}|\y''(t)|^2\,e^{-ht}\,dt-
\int_0^{+\infty}\mathbf f_h(t)\cdot\y(t)\,e^{-ht}\,dt\quad \text{if}\ \y \in\mathcal A_h\\
&\\
& +\infty\quad \text{otherwise in}\ W^{2,1}_{loc}(\R^+;\R^N)
\end{array}\right.
\]
over $\mathcal A_h$, where
\[ \mathcal A_h:=\left\{\y\in W^{2,1}_{loc}(\R^+;\R^N): \ \int_0^{+\infty}|\y''(t)|^2e^{-ht}\,dt < +\infty,\ \y(0)=\u_0,\
\y'(0)=\v_0\right\}.
\]
\end{lemma}
\begin{proof}
 Since $\overline \u_h\in\mathcal A$ and $\overline\u_h(0)=\u_0$, $\overline \u_h'(0)=h^{-1}\v_0$, we directly see that $\overline \y_h\in\mathcal A_h$ and that $ h^{-1}\mathcal F_h(\overline\y_h)=\mathcal J_h(\overline\u_h)$. \KKK
 Moreover, if $\y\in \mathcal A_h$, by setting $\u_h(t)=\y(h^{-1}t)$ we get $\u_h\in \mathcal A,\ \u_h(0)=\u_0,\ \u_h'(0)=h^{-1}\v_0$ and  $h^{-1}\mathcal F_h(\y)=\mathcal J_h(\u_h)$. \KKK Therefore $\mathcal F_h(\overline\y_h)\le \mathcal F_h(\y)$ for every $\y\in \mathcal A_h$  and equality holds if and only if $\mathbf y=\overline{\mathbf y}_h$, \KKK as claimed.
\end{proof}

\section{Proof of Theorem \ref{main}}

Given $\overline{\yy}_h$ minimizing $\mathcal F_h$ over $\mathcal A_h$, here we prove suitable boundedness estimates for the sequence $(\overline{\mathbf y}_h)_{h\in\mathbb N}$, which is the main step towards the proof of Theorem \ref{main}. \KKK

\begin{lemma}\lab{equibound} For every $h\in\mathbb N$, let $\overline \y_h$ as in {\rm Lemma \ref{scaling}}. Then ${\overline \y}_h''\in L^{\infty}(\R^+;\R^N)$ and 
\[
\|{\overline\y}_h''\|_{\infty}\le m^{-1}\sup_{h\in\mathbb N}\|\mathbf f_h\|_{\infty}.
\]
Moreover, the sequence $(\overline \y_h)_{h\in\mathbb N}$ is equibounded in $W^{2,\infty}((0,T); \R^N)$ for every $T > 0$.
\end{lemma}
\begin{proof} Let $h\in\mathbb N$, $\pphi\in C_c(\R^+;\R^N)$ and let $\xxi$ be the unique solution to 
\[\left\{\begin{array}{ll}& \xxi''=e^{t}\pphi,\qquad t>0,\\
&\\
& \xxi (0)=\xxi' (0)=0.\\
\end{array}
\right.
\]
 By setting  $\ppsi_h(t):= h^{-2}\xxi(ht)$ \KKK
 we see that $\ppsi_h(0)=\ppsi'_h(0)=0$ and that 
 \[\int_0^{+\infty}|\ppsi_h''(t)|^2e^{-ht}\,dt=h^{-1}\int_0^{+\infty}|\pphi(t)|^2e^t\,dt,\]
 and the integral in the right hand side is finite since $\pphi\in C_c(\R^+;\R^N)$, thus \KKK
  we get $\overline\y_h+\ppsi_h\in \mathcal A_h$. The minimality of $\overline\y_h$ entails the validity of the first order relation
\beeq\lab{est10}
mh^{-2}\int_0^{+\infty}{\overline\y_h''}(t)\cdot \ppsi_h''(t)\,e^{-ht}\,dt=\int_0^{+\infty}\mathbf f_h(t)\cdot\ppsi_h(t) e^{-ht}\,dt.
\eneq
 Since   $\xxi (0)=0$, using integration by parts  we have
for every $\nu>0$ and every $\tau>0$
\[
\begin{aligned}
\int_0^\tau |\xxi(t)|e^{-t}\,dt&\le\int_0^\tau\sqrt{|\xxi(t)|^2+\nu^2}e^{-t}\,dt=[-e^{-t}\sqrt{|\xxi(t)|^2+\nu^2}]_0^\tau+\int_0^\tau\frac{\xxi'(t)\cdot\xxi(t)}{\sqrt{|\xxi(t)|^2+\nu^2}}\,e^{-t}\,dt\\
&\le \nu+\int_0^{\tau}|\xxi'(t)|e^{-t}\,dt,
\end{aligned}
\]
and then by the arbitrariness of $\nu$ and $\tau$, and by repeating the same argument taking into account that $\xxi'(0)=0$, we obtain 
\[
\int_0^{+\infty} |\xxi(t)|e^{-t}\,dt\le\int_0^{+\infty}|\xxi'(t)|e^{-t}\,dt\le \int_0^{+\infty}|\xxi''(t)|e^{-t}\,dt.
\] \KKK
Therefore, 
\beeq\lab{est11}\begin{aligned}
&\displaystyle\left |\int_0^{+\infty}\mathbf f_h(t)\cdot\ppsi_h(t) e^{-ht}\,dt\right |=h^{-3}\left |\int_0^{+\infty}\mathbf f_h(h^{-1}s)\cdot\xxi(s) e^{-s}\,ds\right |\\
&\qquad\displaystyle\le h^{-3}\|\mathbf f_h\|_{\infty}\int_0^{+\infty}e^{-s}|\xxi''(s)|\,ds=h^{-3}\|\mathbf f_h\|_{\infty}\int_0^{+\infty}|\pphi(s)|\,ds.
\end{aligned}
\eneq
We recall from Lemma \ref{scaling} that $\overline\y_h(t)=\overline\u_h(ht)$, where $\overline \u_h$ is the unique solution to \eqref{pb1}. 
 Hence, by taking into account that
\[\begin{aligned}
&\displaystyle h^{-2}\int_0^{+\infty}{\overline\y_h''(t)}\cdot \ppsi_h''(t)\,e^{-ht}\,dt=\int_0^{+\infty}{\overline\u_h''}(ht)\cdot  \xxi'' (ht)\,e^{-ht}\,dt\\
&\qquad\displaystyle =h^{-1}\int_0^{+\infty}{\overline\u_h''(s)}\cdot \xxi'' (s)\,e^{-s}\,ds=h^{-1}\int_0^{+\infty}{\overline\u_h''(s)}\cdot\pphi (s)\,ds,
\end{aligned}
\]
and by using
\eqref{est10} and \eqref{est11}, we get
\beeq\lab{est14}
\displaystyle \left |\int_0^{+\infty}{\overline\u''_h}(s)\cdot \pphi (s)\,ds\right |\le m^{-1}h^{-2}\|\mathbf f_h\|_{\infty}\int_0^{+\infty}|\pphi|\,ds.
\eneq
By the arbitrariness of $\pphi\in C_c(\R^+;\R^N)$, and since $C_c(\R^+;\R^N)$ is dense in $L^1(\R^+;\R^N)$, \eqref{est14} entails
\[
\|{\overline\u''_h}\|_{\infty}\le \frac{1}{h^2m}\|\mathbf f_h\|_{\infty},
\]
that is, 
\beeq\lab{est16}
\|{\overline\y''_h}\|_{\infty}\le \frac{1}{m}\|\mathbf f_h\|_{\infty}.
\eneq

Eventually,  we have for every $t\in [0,T]$
\[ {\overline\y'_h}(t)=\v_0+\int_0^t{\overline\y''_h}(s)\,ds\qquad\mbox{and}\qquad {\overline\y}_h(t)=\u_0 +t\v_0+\int_0^t (t-s){\overline\y''_h}(s)\,ds,\]
hence \eqref{est16} yields
\beeq\label{est17*} \|{\overline\y}_h\|_{L^{\infty}(0,T)}\le |\u_0| +T|\v_0|+ \frac{T^2}{2m}\|\mathbf f_h\|_{\infty}\eneq and
\beeq\label{est18*} \|{\overline\y'_h}\|_{L^{\infty}(0,T)}\le |\v_0|+\frac{T}{m}\|\mathbf f_h\|_{\infty}.\eneq
The estimates \eqref{est16}, \eqref{est17*} and \eqref{est18*} prove the result, since the sequence $(\mathbf f_h)_{h\in\mathbb N}$ is bounded in $L^\infty(\R^+;\R^N)$.
\end{proof}

\noindent
{\bf Proof of Theorem \ref{main}.}\rm\ For every $h\in\mathbb N$,  let $\overline\y_h$ be  as in Lemma \ref{scaling}. Let $T>0$ and let $\xxi\in C^{\infty}(\R)$ with  $\spt \xxi\subset (0,T)$. \KKK Then by setting $\pphi_h(t):=\xxi (t)e^{ht}$ and by taking into account the first order minimality condition  \eqref{est10} we have
\beeq\label{ultima}\begin{aligned}
&\displaystyle -m\int_0^{T}{\overline\y'_h}(t)\cdot (h^{-2}\xxi'''(t)+ 2h^{-1}\xxi''(t)+\xxi'(t))\,dt\\&\qquad\qquad=-mh^{-2}\int_0^{T}{\overline\y'_h}(t)\cdot (\pphi''_h(t)\,e^{-ht})'\,dt\\
&\qquad\qquad\displaystyle =mh^{-2}\int_0^{T}{\overline\y''_h}(t)\cdot \pphi_h''(t)\,e^{-ht}\,dt=mh^{-2}\int_0^{+\infty}{\overline\y''_h}(t)\cdot \pphi_h''(t)\,e^{-ht}\,dt\\
&\qquad\qquad\displaystyle= \int_0^{+\infty}\mathbf f_h(t)\cdot\pphi_h(t) e^{-ht}\,dt=\int_0^{T}\mathbf f_h(t)\cdot\xxi(t)\,dt.
\end{aligned}
\eneq
By Lemma \ref{equibound}   there exists $\overline\y\in W^{2,\infty}((0,T);\R^N)$ such that, up to subsequences, $\overline\y_h\wconv \overline \x$ in $w^*-W^{2,\infty}((0,T);\R^N)$.  Therefore we get $\overline\x(0)=\u_0,\
{\overline\x'}(0)=\v_0$ and by taking into account  \eqref{ultima} and the $w^*-L^\infty(\R^+)$ convergence of $\mathbf f_h$ to $\mathbf f$
  we obtain in the limit as $h\to+\infty$ \[
-m\int_0^T{\overline\y'}(t)\cdot\xxi'(t)\,dt=\int_0^{T}\mathbf f(t)\cdot\xxi(t)\,dt.
\]
 The latter holds for every $\xxi\in C^{\infty}(\R)$ with $\spt \xxi\subset (0,T)$,  \KKK therefore $\overline\x $ is the unique solution of
\[
\left\{\begin{array}{ll}\displaystyle & m{\y''}=\mathbf f\\
&\\
& \y(0)=\u_0,\qquad {\y'}(0)=\v_0
\end{array}\right.
\]
on  $[0,T]$, hence the whole sequence $(\overline{\y}_h)_{h\in\mathbb N}$ is such that $\overline\y_h\wconv \overline \x$ in $w^*-W^{2,\infty}((0,T);\R^N)$. Since the Cauchy problem \eqref{newton0} has a unique solution $\overline\y$ on $\R^+$ and since $T$ is arbitrary,  we conclude that $\overline\y_h\wconv \overline \y$ in $w^*-W^{2,\infty}((0,T);\R^N)$ as $h\to+\infty$ for every $T > 0$ thus proving the theorem.


%

\bigskip


\end{document}